\documentclass{amsart}
\usepackage{forest}
\useforestlibrary{edges}
\usetikzlibrary{arrows.meta,calc}
\usepackage[T1]{fontenc}   
\usepackage{amsfonts, amsmath, amssymb, amscd, amsthm}
\usepackage{graphicx, mathtools,xfrac,url}
\tikzset{-->/.style={dashed,->,arrows=-Latex},
         refract1/.style={to path={-- ++(-0.5,-.5)-- (\tikztotarget)}},
         refract2/.style={to path={-- ($(\tikztotarget)+(120:7mm)$)--(\tikztotarget)}}}

\usepackage{pgfplots}
\usepackage{tikz,tikz-cd,amscd}
\usetikzlibrary{calc}
\usepackage{mathdots,color}
\usetikzlibrary{decorations.pathreplacing,decorations.markings}
\usetikzlibrary{shapes,shapes.arrows,positioning}

\tikzset{->-/.style={decoration={
  markings,
  mark=at position #1 with {\arrow{>}}},postaction={decorate}}}
\tikzset{->-/.default=.5}

\newcounter{alphathm}

\theoremstyle{plain}
\newtheorem{theorem}{Theorem}

\newtheorem{lemma}[theorem]{Lemma}

\newtheorem*{claim*}{Claim}

\theoremstyle{definition}
\newtheorem{definition}[theorem]{Definition}

\newtheorem{example}[theorem]{Example}
\newtheorem{remark}[theorem]{Remark}

\newtheorem{remark-convention}[theorem]{Remark-Convention}

\numberwithin{equation}{section}
\numberwithin{theorem}{section}

\newcommand{\fakeenv}{} 

\newenvironment{restate}[2]  
{ 
 \renewcommand{\fakeenv}{#2} 
 \theoremstyle{plain} 
 \newtheorem*{\fakeenv}{#1~\ref{#2}} 
 \begin{\fakeenv}
}
{
 \end{\fakeenv}
}

\newcommand{\RR}{\mathbb{R}}

\newcommand{\ZZ}{\mathbb{Z}}

\newcommand{\hX}{\widehat X}

\newcommand{\ta}{{\widetilde a}}

\newcommand{\from}{\colon\thinspace}

\newcommand{\param}%
	{{\mathchoice{\mkern1mu\mbox{\raise2.2pt\hbox{$\centerdot$}}\mkern1mu}%
	{\mkern1mu\mbox{\raise2.2pt\hbox{$\centerdot$}}\mkern1mu}%
	{\mkern1.5mu\centerdot\mkern1.5mu}{\mkern1.5mu\centerdot\mkern1.5mu}}}

\begin{document}

\vspace{0.5in}

\renewcommand{\bf}{\bfseries}
\renewcommand{\sc}{\scshape}
\vspace{0.5in}


\title[Pseudo-Isometric Surgery]{Pseudo-Isometric Surgery}

\author[M.~Clay]{Matt Clay}
\address{Department of Mathematical Sciences \\
University of Arkansas\\
Fayetteville, AR 72701}
\email{mattclay@uark.edu}

\author[J.~Thompson]{Josh Thompson}
\address{Department of Mathematics \& Computer Science \\
Northern Michigan University \\
Marquette, MI 49855}
\email{joshthom@nmu.edu}



\begin{abstract} 
We introduce a type of surgery on metric spaces.  This surgery, in some sense, seeks to replace a subspace $S$ of a metric space $X$ with another metric space $T$ via a function $f \from S \to T$.  When $T$ is a discrete space, this amounts to collapsing the subspace according to the function.  This surgery results in a new metric space we denote $\widehat{X}_f$ and there is a natural function $F \from X \to \widehat{X}_f$ induced from $f$.  Our primary interest is investigating if properties of the original function $f$ are inherited by the induced function $F$.  We show that if $f$ is a pseudo-isometry then so is $F$.  However, for a quasi-isometry, a very natural generalization of a pseudo-isometry that is prevalent in geometric group theory, such a result does not hold.

\end{abstract}

\maketitle


\section{Introduction}\label{sec:intro}


The idea of removing a subset from a space and replacing it with a modified version is one the most basic transformations of mathematics.  For example the M\"obius band, often obtained as the result of a cut/twist/reglue operation can also arise from a remove/alter/replace operation on the annulus, see Figure~\ref{fig:annulus}.   
Such transformations are used to produce new spaces that are simulatneously different from, yet similar to, the original.  

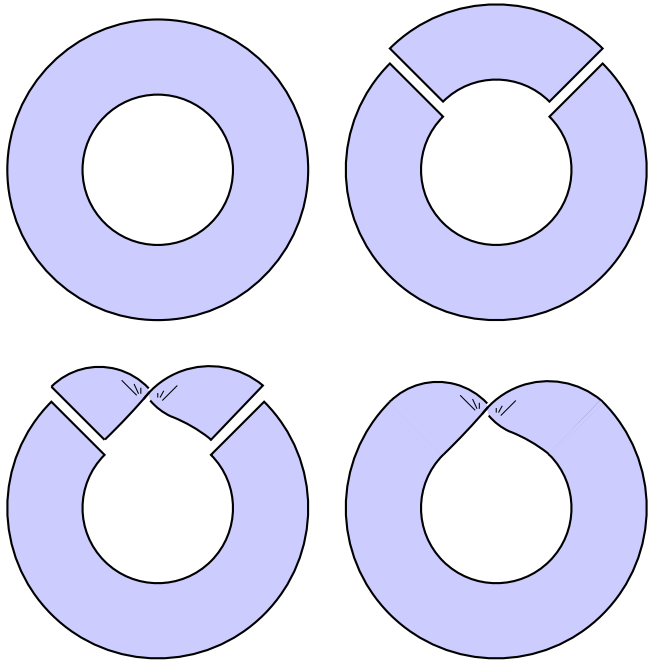
\begin{figure}[ht]
\centering
\begin{tikzpicture}
\def\x{4.5}
\def\xx{9}
\def\xxx{13.5}
\def\a{45}
\def\aa{135}
\def\aaa{-225}
\def\u{.2}
\def\uu{.09}
\def\uuu{.06}
\def\m{44}

\fill [blue!20,even odd rule] (0,0) circle[radius=2cm] circle[radius=1cm];
\foreach \radius in {1,2}
  \draw [thick] (0,0) circle[radius=\radius cm];

\begin{scope}[shift={(\x,0)}]
        \draw[fill=blue!20,thick] ({cos(\a)},{sin(\a)}) coordinate (beta) --  ({2*cos(\a)},{2*sin(\a)}) arc (\a:\aaa:2) coordinate (alpha)  --  ({2*cos(\aa)},{2*sin(\aa)}) -- ({cos(\aa)},{sin(\aa)}) arc (\aa:360+\a:1) -- cycle;
    \end{scope}

\begin{scope}[shift={(\x,.2)}]
        \draw[fill=blue!20,thick] ({cos(\a)},{sin(\a)}) coordinate (beta) --  ({2*cos(\a)},{2*sin(\a)}) arc (\a:\aa:2) coordinate (alpha)  --  ({2*cos(\aa)},{2*sin(\aa)}) -- ({cos(\aa)},{sin(\aa)})  arc (\aa:\a:1) -- cycle;
    \end{scope}

\begin{scope}[shift={(0,-\x)}]
        \draw[fill=blue!20,thick] ({cos(\a)},{sin(\a)}) coordinate (beta) --  ({2*cos(\a)},{2*sin(\a)}) arc (\a:\aaa:2) coordinate (alpha)  --  ({2*cos(\aa)},{2*sin(\aa)}) -- ({cos(\aa)},{sin(\aa)}) arc (\aa:360+\a:1) -- cycle;
    \end{scope}

\begin{scope}[shift={(0,-\x+.2)}]

                            F I L L 
    \fill[blue!20,thick]
    ({2*cos(\aa) - .0075},{2*sin(\aa)}) arc (\aa:43:.9) --  
    ({2*cos(.5*(\aa + \a)) - .19},{1.5*sin(.5*(\aa+\a)) - .2555555}) arc (125:150:1.7) --  
    ({2*cos(\aa)},{2*sin(\aa)}) arc (\aa:43:.9) --   
    ({cos(\aa) + .006},{sin(\aa) + .005}) arc(\aa + 180 - 1 : \aa + 180 + 4.5: 9) -- 
    ({2*cos(.5*(\aa + \a)) - .22},{1.5*sin(.5*(\aa+\a))});

    \fill[blue!20,thick]
    ({cos(\aa) + .006},{sin(\aa) + .005}) arc(\aa + 180 - 1 : \aa + 180 + 4.5: 9) -- 
    ({2*cos(.5*(\aa + \a)) - .22},{1.5*sin(.5*(\aa+\a))}); 

    \fill[blue!20, thick] 
    (-.1, 1.23) arc(220:245:.8) -- 
    (0.2,1.01) arc(68:52:2) --
    ({cos(\a)},{sin(\a)}) coordinate (beta) -- 
    ({2*cos(\a) - .0165},{2*sin(\a) + .016})  arc (50:138:1.1); 

                            D R A W 

    \def\ll{.1}
    \draw[thick] 
    (-.1, 1.23) arc(220:245:.8) -- 
    (0.2,1.01) arc(68:52:2) --
    ({cos(\a)},{sin(\a)}) coordinate (beta) -- 
    ({2*cos(\a) - .0165},{2*sin(\a) + .016})  arc (50:143:1.1); 

    \draw[thick] ({cos(\aa) + .006},{sin(\aa) + .005}) arc(\aa + 180 - 1 : \aa + 180 + 4.5: 9); 

    \draw[thick]
    ({2*cos(\aa)},{2*sin(\aa)}) arc (\aa:43:.9);  

    \draw[thick]
    ({2*cos(\aa)},{2*sin(\aa)}) --  ({cos(\aa)},{sin(\aa)});  

    \def\dm{1}
    \draw (.06, 1.23)--  ++(\u*\dm,\u);
    \draw (.03, 1.27) -- ++(\uu*\dm - .04,\uu);
    \draw (0, 1.27) --   ++(\uuu*\dm - .06,\uuu);

    \draw (-.48, 1.50) -- ++(\u*\dm,-\u);
    \draw (-.315, 1.45) -- ++(\uu*\dm - .03, -\uu - .04);
    \draw (-.225, 1.4) -- ++(\uuu*\dm - .06,-\uuu - .02);

\end{scope}

\begin{scope}[shift={(\x,-\x)}]
        \draw[fill=blue!20,thick] ({cos(\a)},{sin(\a)}) coordinate (beta) --  ({2*cos(\a)},{2*sin(\a)}) arc (\a:\aaa:2) coordinate (alpha)  --  ({2*cos(\aa)},{2*sin(\aa)}) -- ({cos(\aa)},{sin(\aa)}) arc (\aa:360+\a:1) -- cycle;
    \end{scope}

\begin{scope}[shift={(\x,-\x)}]

                            F I L L 
    \fill[blue!20,thick]
    ({2*cos(\aa) - .0075},{2*sin(\aa)}) arc (\aa:43:.9) --  
    ({2*cos(.5*(\aa + \a)) - .19},{1.5*sin(.5*(\aa+\a)) - .2555555}) arc (125:150:1.7) --  
    ({2*cos(\aa)},{2*sin(\aa)}) arc (\aa:43:.9) --   
    ({cos(\aa) + .006},{sin(\aa) + .005}) arc(\aa + 180 - 1 : \aa + 180 + 4.5: 9) -- 
    ({2*cos(.5*(\aa + \a)) - .22},{1.5*sin(.5*(\aa+\a))});

    \fill[blue!20,thick]
    ({cos(\aa) + .006},{sin(\aa) + .005}) arc(\aa + 180 - 1 : \aa + 180 + 4.5: 9) -- 
    ({2*cos(.5*(\aa + \a)) - .22},{1.5*sin(.5*(\aa+\a))}); 

    \fill[blue!20, thick] 
    (-.1, 1.23) arc(220:245:.8) -- 
    (0.2,1.01) arc(68:52:2) --
    ({cos(\a)},{sin(\a)}) coordinate (beta) -- 
    ({2*cos(\a) - .0165},{2*sin(\a) + .016})  arc (50:138:1.1); 

                            D R A W 

    \def\ll{.1}
    \draw[thick] 
    (-.1, 1.23) arc(220:245:.8) -- 
    (0.2,1.01) arc(68:52:2) --
    ({cos(\a)},{sin(\a)}) coordinate (beta) -- 
    ({2*cos(\a) - .0165},{2*sin(\a) + .016})  arc (50:143:1.1); 

    \draw[thick] ({cos(\aa) + .006},{sin(\aa) + .005}) arc(\aa + 180 - 1 : \aa + 180 + 4.5: 9); 

    \draw[thick]
    ({2*cos(\aa)},{2*sin(\aa)}) arc (\aa:43:.9);  

    \def\dm{1}
    \draw (.06, 1.23)--  ++(\u*\dm,\u);
    \draw (.03, 1.27) -- ++(\uu*\dm - .04,\uu);
    \draw (0, 1.27) --   ++(\uuu*\dm - .06,\uuu);

    \draw (-.48, 1.50) -- ++(\u*\dm,-\u);
    \draw (-.315, 1.45) -- ++(\uu*\dm - .03, -\uu - .04);
    \draw (-.225, 1.4) -- ++(\uuu*\dm - .06,-\uuu - .02);

                            F I L L 
   
    \def\mm{.01}
    \def\mm{.01}
    \def\af{.5}
    \draw[blue!20,very thick] ({cos(\aa)-\mm},{sin(\aa)+\mm}) -- ({2*cos(\aa)+\mm},{2*sin(\aa)-\mm}) ;
    \draw[blue!20,very thick] ({cos(\a)+\mm},{sin(\a)+\mm}) -- ({2*cos(\a)-\mm},{2*sin(\a)-\mm}) ;
    \draw[blue!20,very thick] ({cos(\a)+\mm},{sin(\a)+\mm}) -- ({2*cos(\a + \af)-\mm},{2*sin(\a + \af)-\mm}) ;
    \end{scope}

\end{tikzpicture}
\caption{The M\"obius band is obtained from a remove/surgery on the annulus.}
\label{fig:annulus}
\end{figure}


In 1910  Max Dehn introduced a procedure in three dimensions, later referred to as "surgery" by Milnor and Thom \cite{Milnor2} and now known as \emph{Dehn Surgery}.  In it, one first removes a solid torus $T$ from a 3--manifold and then `sews it back differently', see \cite{Gordon} and \cite{Stillwell} for details.  There are many different ways to sew the solid torus back in.  Specifically, note that simple closed curves on the torus can be identified with ordered pairs of relatively prime integers $(m,n)$ corresponding to how the curve winds around the surface.  Once the solid torus $T$ is removed we can glue it back in so that the  $(m,n)$ curve is sewn to the curve that previously matched the $(1,0)$ curve.  Each different choice of $m$ and $n$ results in a potentially different 3--manifold.  In fact, every closed, orientable, connected 3--manifold can be obtained by Dehn surgery on a collection of solid tori in the 3-sphere; a result known as the Lickorish--Wallace theorem \cite{Rolfsen}.   


One might also examine the impact of surgery on the underlying geometry of a space.  Consider $\RR$ under the standard metric.  For this surgery, instead of `sewing it back differently' we `sew in something else' as follows: remove every interval of the form $[2n, 2n+1]$ where $n \in \ZZ$  and replace each with a single point, see Figure~\ref{fig:intro-ex}.  
The resulting space $Y$ with the obvious path metric is isometric to $\RR$. 
This surgery is clearly not an isometry but, as will become clear later, it is a \textit{quasi-isometry} (in fact also a pseudo-isometry too).   

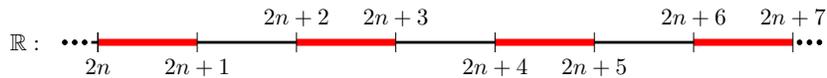
\begin{figure}[ht]
\scalebox{.88}{
\begin{tikzpicture}
\def\s{1.75}
\node at (-1,0) {$\RR:$};
\draw[very thick] (0,0) -- (6*\s,0);
\filldraw (6*\s + 0.2,0) circle [radius=1pt];
\filldraw (6*\s + 0.35,0) circle [radius=1pt];
\filldraw (6*\s + 0.5,0) circle [radius=1pt];

\filldraw (-.6+ 0.2,0) circle [radius=1pt];
\filldraw (-.6 + 0.35,0) circle [radius=1pt];
\filldraw (-.6 + 0.5,0) circle [radius=1pt];
\draw[line width=3pt,red] (0.1,0) -- (1.5+0.1,0) (3+0.1,0)  -- (4.5+0.1,0) (6+0.1,0)  --  (7.5+0.1,0) (9+0.1,0)  --  (10.5+0.1,0);
\draw (1.6,-0.15) -- (1.6,0.15) (0.1,-0.15) -- (0.1,0.15);
\draw (3.1,-0.15) -- (3.1,0.15) (0.1,-0.15) -- (0.1,0.15);
\draw (4.6,-0.15) -- (4.6,0.15) (0.1,-0.15) -- (0.1,0.15);
\draw (6.1,-0.15) -- (6.1,0.15) (10.6,-0.15) -- (10.6,0.15);
\draw (7.6,-0.15) -- (7.6,0.15) (9.1,-0.15) -- (9.1,0.15);
\node at (0+0.1,0) [label=below:{$2n$}] {};
\node at (1.5+0.1,0) [label=below:{$2n+1$}] {};
\node at (3+0.1,0) [label=above:{$2n+2$}] {};
\node at (4.5+0.1,0) [label=above:{$2n+3$}] {};
\node at (6+0.1,0) [label=below:{$2n+4$}] {};
\node at (7.5+0.1,0) [label=below:{$2n+5$}] {};
\node at (9+0.1,0) [label=above:{$2n+6$}] {};
\node at (10.5+0.1,0) [label=above:{$2n+7$}] {};
\end{tikzpicture}
    }
\caption{Replacing the red sets by points gives a space quasi-isometric to $\RR$.}
\label{fig:intro-ex}
\end{figure}

We formalize this notion of surgery for an arbitrary metric space $X$ using a map $f \from S \to T$ where $S$ is a subspace of $X$ and $T$ is another metric space.  The question we ask is whether properties of the (local) map used to sew the space $T$ onto $S$ are inherited by the (global) natural map of the total space to the surgered space.  The specific properties we examine in this paper are that of being a quasi-isometry or a pseudo-isometry.  


Quasi-isometries are a very important class of functions on metric spaces that allow for a controlled distortion.  Such a function is a transformation between metric spaces that distorts distances by a uniformly bounded amount, above a given scale.  A precise definition is given in Definition~\ref{def:pseudo qi} and its subsequent remark.  Implicit in the work of $\check{\mbox{S}}$varc \cite{Svarc} in 1955 and Milnor  \cite{Milnor} in 1968, the notion of a quasi-isometry is central to Gromov's idea of \textit{coarse equivalence} of metric spaces.  In 1981, Gromov \cite{Gromov} defined quasi-isometry the way it is used today.  A standard example of a quasi-isometry is the (discontinuous) map that sends each real number to the greatest integer less than or equal to it \cite{Clay}.  For a general introduction to quasi-isometries see \cite{Brid-Helf}.  

Not all quasi-isometries are discontinous, such as the map that collapses the unit interval in $\RR$ to the origin and then scales everything by a factor of two as well as the example illustrated above.  These maps are also examples of a \textit{pseudo-isometry}, a term introduced by Mostow \cite{Mostow} in 1974 in his study of arbitrary symmetric spaces (see also \S 5.9 of \cite{Thurston}).  A pseudo-isometry satisfies a stronger condition than a quasi-isometry in that it has no additive term on the upper bound.  Indeed, a pseudo-isometry is a Lipschitz map that distorts distances by a uniformly bounded amount, above a given scale.  Not all continuous quasi-isometries are pseudo-isometries. An explicit example is given in Example~\ref{ex:strong}.

Our main result is to show that a surgery specified by a pseudo-isometry yields a natural map from the original space to surgered space that is also a pseudo-isometry. 

\begin{theorem}\label{thm:main}
Suppose $(X,d_X)$ and $(T,d_T)$ are metric spaces and consider a subset $S \subset X$ as a metric space with metric induced from $X$.  If $f \from S \to T$ is a pseudo-isometry, then the natural map $F \from X \to \hX_f$ is a pseudo-isometry as well.
\end{theorem}

The definition of $\widehat{X}_f$ appears in Section~\ref{sec:Xf}.  The proof of Theorem~\ref{thm:main} appears in Section~\ref{sec:main}.

The metric on $\hX_f$, obtained by variation of the quoteint metric space construction (see e.g., \cite[Chapter~I.5]{Brid-Helf}), is defined via certain alternating sequences of pairs of points in $X$ and $T$ taking into account the function $f$.  We call such sequences \emph{admissible} and the length of such is the sum of the distances for each pair (Definition~\ref{def:admissible seq}).  The main technical step to prove Theorem~\ref{thm:main} is Lemma~\ref{lem:admissible length lower bound} where we give a lower bound on the length of an admissible sequence in terms of the distance in $X$ between its endpoints.  It is in this lemma that we need to restrict to pseudo-isometries as opposed to quasi-isometries to control the amount of additive error.  

One example that is covered by Theorem~\ref{thm:main} is the map $F \from \RR \to \RR$ that collapses each interval of the form $[2n,2n+1]$ to a point. This is the example discussed above.  For this example we have that
$S = \{[2n,2n+1] \mid n \in \ZZ\}$, $T = \ZZ$ and $f(s) = \lfloor{\frac{s}{2}}\rfloor$.  In this case as we mentioned above, the surgered space is isometric to $\RR$.  

We present several examples in Section~\ref{sec:examples}.  First, an application of Theorem~\ref{thm:main} to regular trees is given in Example~\ref{ex:tree}.  This shows that a regular tree of degree four is pseudo-isometric to a regular tree of degree six, and we show how this generalizes to other degrees.  Next, in Example~\ref{ex:all-trees}, we use a different mapping and Theorem~\ref{thm:main} to show that every regular tree of degree $n$ is pseudo-isometric to a regular tree of degree three. Finally, in Example~\ref{ex:strong} we show the "pseudo-" assumption is necessary in the following sense:  When the gluing map is weakened to that of a quasi-isometry the natural map to the surgered space fails to be a quasi-isometry.  It remains open under what conditions does a quasi-isometric gluing map yield a quasi-isometric natural map between the original and the surgered space.

\subsection{Acknowledgements}
We would like to thank the referees for a careful reading of this work.  

\section{Construction of the surgered space.} \label{sec:Xf}

In this Section we define the surgered space $\widehat{X}_f$ using a notion of \emph{admissible sequences} (Definition~\ref{def:admissible seq}) which ties together the spaces $X$, $S$, and $T$  via the pseudo-isometry $f \from S \to T$.  We also present a few properties of admissible sequences that form the essential parts of the proof of Theorem~\ref{thm:main}.

To begin, we state the definition of a pseudo-isometry.

\begin{definition}\label{def:pseudo qi}
Let $(S,d_S)$ and $(T,d_T)$ be metric spaces.  A map $f \from S \to T$ is a \emph{pseudo-isometry} if there exist contants $K \geq 1$ and $C \geq 0$ such that the following hold.
\begin{enumerate}
\item For all $x_0,x_1 \in S$, we have:
\begin{equation*}
\frac{1}{K}d_S(x_0,x_1) - C \leq d_T(f(x_0),f(x_1)) \leq Kd_S(x_0,x_1).
\end{equation*}
\item For all $y \in T$, there is an $x \in S$ with $d_T(f(x),y) \leq C$.
\end{enumerate}
\end{definition}

\begin{remark}
If we allow the upper bound to also have an additive  constant, i.e.,
\[d_T(f(x_0),f(x_1)) \leq Kd_S(x_0,x_1) + C\]
then the map is called a \emph{quasi-isometry}.
\end{remark}


A metric on the surgered space will be defined via sequences of pairs of points that (potentially) intersect the sets where the surgery occurs, called \textit{admissible sequences}.  In what follows $X$ and $T$ are metric spaces with metrics $d_X$ and $d_T$ respectively, and $S \subseteq X$ is a subspace considered as a metric space with the metric induced from $X$.  We also have a pseudo-isometry $f \from S \to T$ with constants $K$ and $C$ as in Definition~\ref{def:pseudo qi}.

\begin{definition}\label{def:admissible seq}
An \emph{admissible sequence} is a sequence of pairs of the form:
\begin{equation}\label{eq:admissbile path}
\gamma : (x_0,y_1), (u_1,v_1), (x_1,y_2), \ldots ,(u_k,v_k), (x_k,y_{k+1})
\end{equation}
where:
\begin{enumerate}
\item $x_0,y_{k+1} \in X$,
\item $x_i,y_i \in S$ for $i = 1,\ldots,k$,
\item $u_i,v_i \in T$ for $i = 1,\ldots,k$, and
\item $u_i = f(y_i)$ and $v_i = f(x_i)$ for $i = 1,\ldots,k$.  
\end{enumerate}
    We allow for the possibility that $x_i = y_{i+1}$ or $u_i = v_i$ for each $i = 0,\ldots,k$.  Moreover, we allow for the possibility that the pair $(x_0,y_1)$ is omitted.  In this case, the only restriction on $u_1$ is that it lies in $T$.  Likewise, we allow for the possibility that the pair $(x_k,y_{k+1})$ is omitted.  In this case, the only restriction on $v_k$ is that it lies in $T$.  We say the sequence is from $x_0$ to $y_{k+1}$, modifying to use $u_1$ or $v_k$ accordingly if the pair $(x_0,y_1)$ or $(x_k,y_{k+1})$ respectively is omitted. A schematic for an admissible sequence appears in Figure~\ref{fig:adm}.
\end{definition}


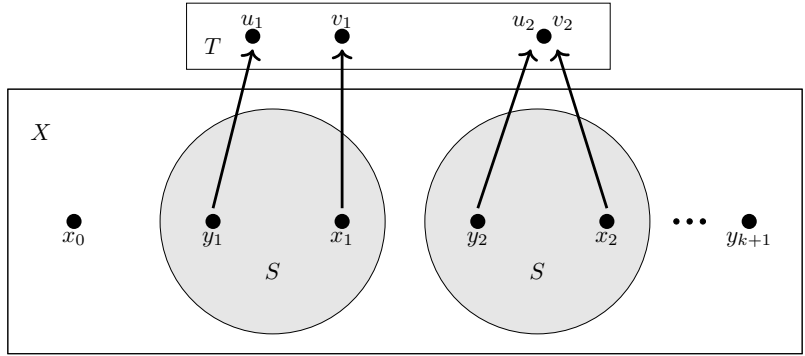
\begin{figure}[ht]
\scalebox{.88}{
\begin{tikzpicture}

\def\xl{.2} 
\def\xt{2.9} 
\def\yt{4.8} 
\def\xc{4} 
\def\yc{2} 
\def\cs{4} 
\def\rx{12} 
\def\ry{4} 
\def\ll{2.35} 
\def\xo{3} 
\def\xa{.2} 
\def\ta{.3} 
\def\to{-.6} 
\def\rr{1.7} 
\def\xx{1} 
\def\ex{.15}
\def\tx{.5}
\def\ys{-1} 

\draw (0,0) rectangle (\rx,\ry);
\coordinate (X) at (.5, \ry - .4);
        \node[below] at (X) {$X$};

\draw[thin, fill=gray!20] (\xc,\yc) circle [radius=\rr];
\draw[thin, fill=gray!20] (\xc+\cs,\yc) circle [radius=\rr];
    \coordinate (S) at (\xc,\yc-.5);
        \node[below] at (S) {$S$};
    \coordinate (S1) at (\xc + \cs,\yc-.5);
        \node[below] at (S1) {$S$};

\draw (0,0) rectangle (\rx,\ry);
    \draw (\xt + \xl - \to - \tx - .5,\yt - \tx) rectangle (\xt + \xl + \cs + \xx + \tx + .5,\yt + \tx);
    \coordinate (T) at (\xt + \xl - \to - \tx - .1,\yt - \tx + .1);
        \node[above] at (T) {$T$};

    \coordinate (x0) at (\xc - \xo,\yc);
        \node[below,yshift=\ys] at (x0) {$x_0$};
            \draw[thick, fill=black] (x0) circle [radius=.1];

    \coordinate (y1) at (\xt + \xl,\yc);
        \node[below,yshift=\ys] at (y1) {$y_1$};
            \draw[thick, fill=black] (y1) circle [radius=.1];
            
    \coordinate (u1) at (\xt + \xl - \to,\yt);
        \node[above] at (u1) {$u_1$};
            \draw[thick, fill=black] (\xt + \xl - \to,\yt) circle [radius=.1];

    \draw [very thick,  ->] (\xt + \xl,\yc + \xa) -- (\xt + \xl - \to,\yt - \xa);
    
    \coordinate (v1) at (\xt + \ll - \xl,\yt);
        \node[above] at (v1) {$v_1$};
            \draw[thick, fill=black] (\xt + \ll- \xl,\yt) circle [radius=.1];

    \coordinate (x1) at (\xt + \ll - \xl,\yc);
        \node[below,yshift=\ys] at (x1) {$x_1$};
            \draw[thick, fill=black] (\xt + \ll- \xl,\yc) circle [radius=.1];

    \draw [very thick, <-]  (\xt  + \ll - \xl,\yt - \xa) -- (\xt + \ll - \xl ,\yc + \xa);

    \coordinate (y2) at (\xt + \xl + \cs,\yc);
        \node[below,yshift=\ys] at (y2) {$y_2$};
            \draw[thick, fill=black] (\xt + \xl + \cs,\yc) circle [radius=.1];

    \coordinate (u1) at (\xt + \xl + \cs + \xx,\yt + \xl);
        \node[left] at (u1) {$u_2$};
        
    \draw [very thick, ->] (\xt + \xl + \cs ,\yc + \xa) -- (\xt + \xl  + \cs + \xx - .2,\yt - \xa) ;
 
    \coordinate (v2) at (\xt + \xl + \cs + \xx,\yt + \xl);
        \node[right] at (v2) {$v_2$};
            \draw[thick, fill=black] (\xt + \xl + \cs + \xx,\yt) circle [radius=.1];

    \coordinate (x2) at (\xt + + \ll - \xl + \cs,\yc);
        \node[below,yshift=\ys] at (x2) {$x_2$};
            \draw[thick, fill=black] (\xt + \ll- \xl + \cs,\yc) circle [radius=.1];

    \draw [very thick, <-]  (\xt + \xl + \cs + \xx + .2 ,\yt - \xa) -- (\xt + \ll - \xl + \cs ,\yc + \xa)  ;

    \coordinate (y_k+1) at (\xc + \cs + \xo + .2,\yc);
        \node[below,yshift=\ys] at (y_k+1) {$y_{k+1}$};
            \draw[thick, fill=black] (\xc + \cs + \xo + .2,\yc) circle [radius=.1];

    \draw[thin, fill=black] (\xc + \ll - \xl + \cs + \ex,\yc) circle [radius=.04];
    \draw[thin, fill=black] (\xc + \ll - \xl + \cs + \ex + .2,\yc) circle [radius=.04];
    \draw[thin, fill=black] (\xc + \ll - \xl + \cs + \ex - .2,\yc) circle [radius=.04];

\end{tikzpicture}
}
\caption{An admissible sequence}
\label{fig:adm}
\end{figure}


\begin{definition}\label{def:pseudo-metric}
The \emph{length} of an admissible sequence $\gamma$ as defined in~\eqref{eq:admissbile path} is:
\begin{equation}\label{eq:admissible seq length}
\ell(\gamma) = d_X(x_0,y_1) + \sum_{i=1}^k \bigl(d_T(u_i,v_i) + d_X(x_i,y_{i+1})\bigr).
\end{equation}
\end{definition}


The next lemma shows that the length of an admissible sequence between points $x$ and $y$ in $X$ is bounded below by a linear function of the distance in $X$ between $x$ and $y$.  This lemma is essential to the proof of Theorem~\ref{thm:main} as it forms the basis of the proving the pseudo-isometry inequalities.  

\begin{lemma}\label{lem:admissible length lower bound}
    Let $\gamma$ be an admissible sequence from $x$ to $y$, where $x,y \in X$.  Then \[d_X(x,y) \leq K^2\ell(\gamma) + KC .\] 
\end{lemma}

\begin{proof}
Let the admissible sequence $\gamma$ be given by:
\begin{equation*}
\gamma : (x_0,y_1), (u_1,v_1), (x_1,y_2), \ldots ,(u_k,v_k), (x_k,y_{k+1})
\end{equation*}
where $x_0 = x$ and $y_{k+1} = y$. 

By the definition of an admissible sequence, we have $f(x_i) = v_i$ and $f(y_i) = u_i$ for $i = 1,\ldots,k$.  The assumption that $f$ is a pseudo-isomtery implies 
\begin{equation*}
d_X(y_1,x_k) \leq Kd_T(u_1,v_k) + KC \text{ and } d_T(v_i,u_{i+1}) \leq Kd_X(x_i,y_{i+1}).
\end{equation*}
Combining the triangle inequality with the first inequality gives~\eqref{eq:1},~\eqref{eq:2} and~\eqref{eq:3}.  Regrouping gives~\eqref{eq:4} and the second of these inequalities is used in ~eqref{eq:5}.  As $K \geq 1$, further rearranging gives~\eqref{eq:6}.  Finally,~\eqref{eq:7} follows from the definition of the length of $\gamma$. 
\begin{align}
    d_X(x,y) & \leq d_X(x_0,y_1) + d_X(y_1,x_k) + d_X(x_k,y_{k+1}) \label{eq:1} \\
    & \leq d_X(x_0,y_1) + (Kd_T(u_1,v_k) + KC) + d_X(x_k,y_{k+1}) \label{eq:2}  \\
    & \leq d_X(x_0,y_1) + K\left(\sum_{i=1}^{k-1} \bigl(d_T(u_i,v_i) + d_T(v_i,u_{i+1}) \bigr) + d_T(u_k,v_k)\right)  \nonumber \\
    & \hspace{2.19cm} +  d_X(x_k,y_{k+1}) + KC  \label{eq:3} \\
    & = d_X(x_0,y_1) + K\left(\sum_{i=1}^{k} d_T(u_i,v_i) + \sum_{i=1}^{k-1} d_T(v_i,u_{i+1})\right) \nonumber  \\
    & \hspace{2.19cm} + d_X(x_k,y_{k+1}) + KC  \label{eq:4} \\
    & \leq d_X(x_0,y_1) + K\left(\sum_{i=1}^{k} d_T(u_i,v_i) + K\sum_{i=1}^{k-1} d_X(x_i,y_{i+1})\right) \nonumber \\ 
    & \hspace{2.19cm} + d_X(x_k,y_{k+1}) + KC  \label{eq:5} \\
    & \leq K^2\left( d_X(x_0,y_1) + \sum_{i=1}^k \bigl(d_T(u_i,v_i) + d_X(x_i,y_{i+1})\bigr)\right) + KC  \label{eq:6} \\
    & \leq K^2 \ell(\gamma) + KC.  \label{eq:7}
\end{align}
This completes the proof of the lemma.
\end{proof}

\begin{remark}\label{rmk:admissible length lower bound}
    Note that in the proof above it is necessary that $f$ is a pseudo-isometry and not merely a quasi-isometry.  In passing from~\eqref{eq:4} to~\eqref{eq:5}, the $d_X$ summation has no additive term, effectively allowing us to bound the lengths with a multiplicative constant.  Had $f$ been just a quasi-isometry this summation would induce $k-1$ additive constants.  The number of such constants reflects the number of steps in the admissible sequence which is not bounded by the distance. This makes it impossible to bound the distance between $x$ and $y$ in terms of the length of an admissible sequence between them.  
\end{remark}

As we complete the construction of the surgered space let us recall the orginal space $X$, a subset $S \subset X$ and a pseudo-isometry $f \from S \to T$.  We first glue $S$ to $T$ forming the space $X'$: 
\begin{equation*}
    X' = \raisebox{3pt}{$X \cup T$} \Big/ \raisebox{-3pt}{$s \sim f(s), \, \,  \forall s \in S$}.
\end{equation*} 
In other words, points in $X'$ are equivalence classes.  Let $j \from X \cup T \to X'$ be the quotient map that takes a point to its equivalence class.  These equivalence classes are one of three types:
\begin{enumerate}
\item If $x \in X - S$, then $j(x) = \{x\}$, a singleton,
\item If $x \in S$, then $j(x) = \{y \in S \mid f(y) = f(x) \} \cup \{ f(x)\}$, or
\item If $u \in T$, then $j(u) = \{y \in S \mid f(y) = u \} \cup \{ u \}$.
\end{enumerate} 
Note, the first set in the union for type (3) may be empty.

The infimum of lengths of admissible sequences induces a pseudo-metric $p_{X'} \from X' \times X' \to \RR$  defined by:
\begin{equation*}
    p_{X'}(x',y') = \inf \{\ell(\gamma) \mid \gamma\} 
\end{equation*} 
where $\gamma$ is an admissible sequence from $x$ to $y$ where $j(x) = x'$ and $j(y) = y'$.  
We define $(\hX_f,d_{\hX_f})$ as the metric space induced by identifying points in $(X',p_{X'})$ that have pseudo-distance equal to 0.  If the corresponding quotient map is $q \from X' \to \hX_f$ we have
\begin{equation*}
d_{\hX_f}(\hat{x},\hat{y}) = \inf\{ p_{X'}(x',y') \mid q(x') = \hat{x} \text{ and } q(y') = \hat{y}\}.
\end{equation*}  
There is an induced map $F \from X \to \hX_f$ given by the composition:
\begin{equation*}
F \from X {\buildrel j \over \to} X' {\buildrel q \over \to} \hX_f.
\end{equation*}
Summarizing the above, we have that $d_{\hX_f}(\hat{x},\hat{y})$ is the infimum of the set of lengths of admissible sequences from a point in $F^{-1}(\hat{x})$ to a point in $F^{-1}(\hat{y})$. The map $F$ can be thought of as a kind of surgery on $X$, in which a subset $S$ is exised and replaced by a set $T$.  

We remark here that an immediate consequence of this definition is that 
\[d_{\hX_f}(q(x'),q(y')) \leq p_{X'}(x',y') \mbox{ } \forall \mbox{ } x',y'  \in X'.\]
This will be used in the proof of Theorem \ref{thm:main}.



The lemma below indicates that if this surgery glues two points together then the two points were a bounded distance apart in the original metric.    

\begin{lemma}\label{lem:same image upper bound}
If $x,y \in X$ and $F(x) = F(y)$, then $d_X(x,y) \leq 3KC$.
\end{lemma}

\begin{proof}
Fix points $x,y \in X$ and suppose that $F(x) = F(y) = \hat{x}$.  Let $x' = j(x)$ and $y' = j(y)$.  By the construction of $\hX_f$, we have that for any two points in $q^{-1}(\hat{x})$ the pseudo-distance is equal to 0.  Hence $p_{X'}(x',y') = 0$. 
Therefore, for any $\epsilon > 0$, there must be points $x_0,y_0 \in X \cup T$ with $j(x_0) = x'$, $j(y_0) = y'$, and an admissible sequence $\gamma$ from $x_0$ to $y_0$ of length less than $\epsilon$.

By the definition of $j \from X \cup T \to X'$ we can assume that $x_0$ and $y_0$ lie in $X$.  Indeed, if $x_0 \in T$ then we must have $x_0 = f(s_0)$ for some $s_0 \in S$ as $j$ is injective on $T - f(S)$.  We could then prepend the sequence $\gamma$ by the ordered pair $(s_0,s_0)$ to get a new sequence with the same properties that starts in $X$.  Likewise if $y_0 \in T$.

    Since $j(x) = j(x_0)$, we must have that $f(x) = f(x_0)$ and hence \[d_X(x,x_0) \leq KC\] since $f$ is a pseudo-isometry.  Likewise we have that $d_X(y,y_0) \leq KC$.  Using Lemma~\ref{lem:admissible length lower bound}, we have that $d_X(x_0,y_0) \leq K^2\epsilon + KC$.  Combining these with the triangle inequality, we find:
\begin{equation*}
d_X(x,y) \leq d_X(x,x_0) + d_X(x_0,y_0) + d_X(y_0,y) \leq KC + K^2\epsilon + KC + KC.
\end{equation*}
As this holds for all $\epsilon > 0$, we have $d_X(x,y) \leq 3KC$ as claimed.
\end{proof}

\section{Proof of Main Theorem}\label{sec:main}

We can now prove the main theorem.  It is restated here for convenience.

\begin{restate}{Theorem}{thm:main}
Suppose $(X,d_X)$ and $(T,d_T)$ are metric spaces and consider a subset $S \subset X$ as a metric space with metric induced from $X$.  If $f \from S \to T$ is a pseudo-isometry, then the natural map $F \from X \to \hX_f$ is a pseudo-isometry as well.
\end{restate}

\begin{proof}
Let $X$, $S$, $T$ be as in the statement of the theorem and $f \from S \to T$ a pseudo-isometry with constants $K \geq 1$ and $C > 0$ as in Definition~\ref{def:pseudo qi}.  Let $X'$ and $\hX_f$ be as defined in Section~\ref{sec:Xf} and $F \from X \to \hX_f$ the map induced by the composition:
\begin{equation*}
F \from X {\buildrel j \over \to} X' {\buildrel q \over \to} \hX_f.
\end{equation*} 

First we show $F$ is coarsely surjective.  By construction, for any $\hat{y} \in \hX_f - F(X)$ there is some $t \in T$ such that $\hat{y} = q(j(t))$.  As $f \from S \to T$ is coarsely surjective, there is an $s \in S$ where $d_T(f(s),t) \leq C$.  The admissible sequence $(s,s), (f(s),t)$ from $s$ to $t$ has length at most $C$.  Therefore $p_{X'}(j(s),j(t)) \leq C$ and hence
\begin{equation*}
d_{\hX_f}(F(s),\hat{y}) = d_{\hX_f}(q(j(s)),q(j(t))) \leq p_{X'}(j(s),j(t)) \leq C
\end{equation*}
as well.

Next we demonstrate an upper bound on $d_{\hX_f}(F(x),F(y))$.  Given points $x,y \in X$, for the admissible sequence $\gamma : (x,y)$ we find that $p_{X'}(j(x),j(y)) \leq \ell(\gamma) = d_X(x,y)$.  Therefore 
\begin{equation*}
d_{\hX_f}(F(x),F(y)) = d_{\hX_f}(q(j(x)),q(j(y))) \leq p_{X'}(j(x),j(y)) \leq d_X(x,y).
\end{equation*}

    Finally, we demonstrate a lower bound on $d_{\hX_f}(F(x),F(y))$.  Let $\epsilon > 0$.  Given points $x,y \in X$, we fix points $x_0, y_0 \in X$ where $F(x) = F(x_0)$, $F(y) = F(y_0)$, and for which there exists an admissible sequence $\gamma$ from $x_0$ to $y_0$ of length less than $d_{\hX_f}(F(x),F(y)) + \epsilon$.  We note that such points and admissible sequence exist by an argument similar to the one presented in Lemma~\ref{lem:same image upper bound}.  

By Lemmas~\ref{lem:admissible length lower bound} and \ref{lem:same image upper bound} we now have that:
\begin{align*}
d_X(x,y) &\leq d_X(x,x_0) + d_X(x_0,y_0) + d_X(y_0,y) \\
&\leq  3KC + K^2\ell(\gamma) + KC + 3KC \\
&\leq K^2\left(d_{\hX_f}(F(x),F(y)) + \epsilon\right) + 7KC. 
\end{align*}
As this holds for all $\epsilon > 0$, rearranging we find that:
\begin{equation*}
\frac{1}{K^2} d_X(x,y) - \frac{7C}{K} \leq d_{\hX_f}(F(x),F(y)).\qedhere
\end{equation*}
\end{proof}

\section{Examples}\label{sec:examples}

In this section we present two infinite families of examples and one example that showcases the need for the stronger pseudo-isometry assumption.   The first is an application of Theorem~\ref{thm:main} where we show that the regular tree of degree 4 is pseudo-isometric to the regular tree of degree 6 (Example~\ref{ex:tree}).  We show how to generalize this using finite graphs and covering space theory.  Next, we present an example that shows that every regular tree of degree at least 3 is pseudo-isometric to the regular tree of degree 3 (Example~\ref{ex:all-trees}).  Lastly, the ``pseudo-isometry'' assumption in Theorem~\ref{thm:main} cannot be weakened to ``quasi-isometry,'' even if the conclusion is also weakened  (Example~\ref{ex:strong}).


\begin{example}\label{ex:tree}
Let $X$ be the regular tree of degree 4 where every edge is isometric to the unit interval $[0,1]$.  We consider the subset $S \subset X$ consisting of every other horizontal edge, including its incident vertices, as indicated in Figure~\ref{fig:tree}.  Specifically, we can identify $X$ as the Cayley graph of the free group of rank 2, $F_2$, corresponding to a basis $\{g_1,g_2\}$ (see Chapter 2 of~\cite{Clay} for details).  Then there are two orbits of edges corresponding to the set of horizontal edges and the set of vertical edges respectively in Figure~\ref{fig:tree}.  Let $e_1$ be the edge whose originating vertex is the identity and whose terminal vertex is $g_1$.  Then the set $S$ corresponds to the set of edges of the form $we_1$ where the terminal syllable of $w$ is an even power of $g_1$.  In other words, $w \in F_2$ is of the form:
\begin{equation*}
w = w'g_1^{2p}
\end{equation*}  
where $w' \in F_2$ is either trivial or ends in $g_2^{\pm 1}$ and $p \in \ZZ$.  

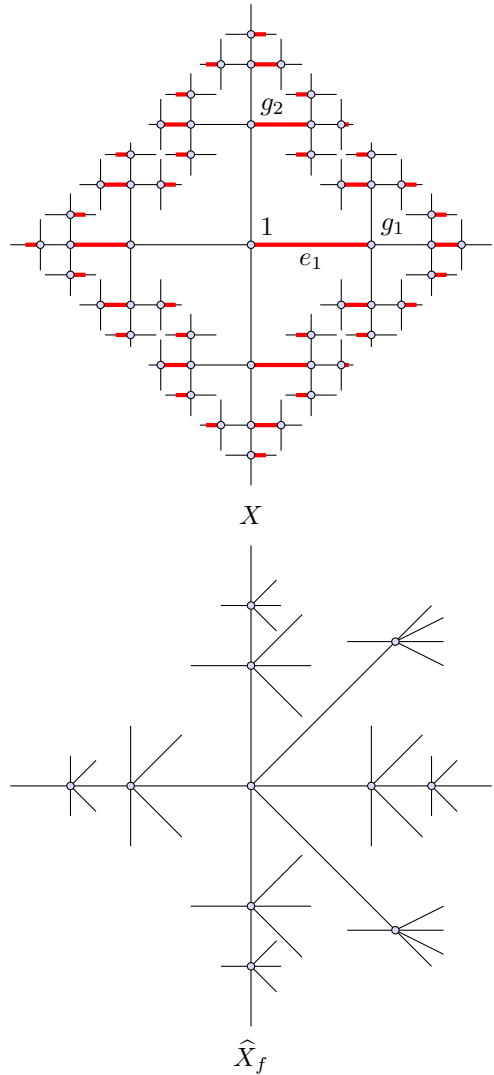
\begin{figure}[ht]
\centering
\begin{tikzpicture}[scale=0.8]
\def\ss{2}
\def\xx{0.85}
\def\rr{1.75}
\begin{scope}
\foreach \a in {0,90,180,270} {
	\draw[rotate=\a] (0,0) -- (2*\ss,0);
	\draw[rotate=\a] (\ss,\xx*\ss) -- (\ss,-\xx*\ss);
	\draw[rotate=\a] (1.5*\ss,0.5*\xx*\ss) -- (1.5*\ss,-0.5*\xx*\ss);
	\draw[rotate=\a] (1.75*\ss,0.25*\xx*\ss) -- (1.75*\ss,-0.25*\xx*\ss);
	\draw[rotate=\a] (\ss-0.25*\xx*\ss,0.75*\ss) -- (\ss+0.25*\xx*\ss,0.75*\ss);
	\draw[rotate=\a] (\ss-0.25*\xx*\ss,-0.75*\ss) -- (\ss+0.25*\xx*\ss,-0.75*\ss);
	\draw[rotate=\a] (\ss-0.5*\xx*\ss,0.5*\ss) -- (\ss+0.5*\xx*\ss,0.5*\ss) (\ss-0.5*\xx*\ss,-0.5*\ss) -- (\ss+0.5*\xx*\ss,-0.5*\ss);	
	\draw[rotate=\a] (1.5*\ss-0.25*\xx*\ss,0.25*\ss) -- (1.5*\ss+0.25*\xx*\ss,0.25*\ss) (1.5*\ss-0.25*\xx*\ss,-0.25*\ss) -- (1.5*\ss+0.25*\xx*\ss,-0.25*\ss);
	\draw[rotate=\a] (0.75*\ss,0.5*\ss-0.25*\xx*\ss) -- (0.75*\ss,0.5*\ss+0.25*\xx*\ss) (1.25*\ss,0.5*\ss-0.25*\xx*\ss) -- (1.25*\ss,0.5*\ss+0.25*\xx*\ss);
	\draw[rotate=\a] (0.75*\ss,-0.5*\ss-0.25*\xx*\ss) -- (0.75*\ss,-0.5*\ss+0.25*\xx*\ss) (1.25*\ss,-0.5*\ss-0.25*\xx*\ss) -- (1.25*\ss,-0.5*\ss+0.25*\xx*\ss);
}
\draw[ultra thick,red] (-1.875*\ss,0) -- (-1.75*\ss,0) (-1.5*\ss,0) -- (-1*\ss,0) (0,0) -- (\ss,0) (1.5*\ss,0) -- (1.75*\ss,0);
\foreach \a in {1,-1} {
	\draw[ultra thick,red,yscale=\a] (-1.5*\ss,0.25*\ss) -- (-1.375*\ss,0.25*\ss) (1.5*\ss,0.25*\ss) -- (1.625*\ss,0.25*\ss);
	\draw[ultra thick,red,yscale=\a] (-1.25*\ss,0.5*\ss) -- (-\ss,0.5*\ss) (-0.75*\ss,0.5*\ss) -- (-0.625*\ss,0.5*\ss) (0.75*\ss,0.5*\ss) -- (\ss,0.5*\ss) (1.25*\ss,0.5*\ss) -- (1.375*\ss,0.5*\ss);
	\draw[ultra thick,red,yscale=\a] (-1.125*\ss,0.75*\ss) -- (-\ss,0.75*\ss) (-0.625*\ss,0.75*\ss) -- (-0.5*\ss,0.75*\ss) (0.375*\ss,0.75*\ss) -- (0.5*\ss,0.75*\ss) (0.875*\ss,0.75*\ss) -- (\ss,0.75*\ss);
	\draw[ultra thick,red,yscale=\a] (-0.75*\ss,\ss) -- (-0.5*\ss,\ss) (0,\ss) -- (0.5*\ss,\ss) (0.75*\ss,\ss) -- (0.8125*\ss,\ss);
	\draw[ultra thick,red,yscale=\a] (-0.625*\ss,1.25*\ss) -- (-0.5*\ss,1.25*\ss) (0.375*\ss,1.25*\ss) -- (0.5*\ss,1.25*\ss);
	\draw[ultra thick,red,yscale=\a] (-0.375*\ss,1.5*\ss) -- (-0.25*\ss,1.5*\ss) (0,1.5*\ss) -- (0.25*\ss,1.5*\ss);
	\draw[ultra thick,red,yscale=\a] (0,1.75*\ss) -- (0.125*\ss,1.75*\ss);
}
\draw[thin, fill=blue!15] (0,0) circle [radius=\rr pt];
\foreach \a in {0,90,180,270} {
	\draw[thin, fill=blue!15,rotate=\a] (\ss,0) circle [radius=\rr pt];
	\draw[thin, fill=blue!15,rotate=\a] (1.5*\ss,0) circle [radius=\rr pt];
	\draw[thin, fill=blue!15,rotate=\a] (1.75*\ss,0) circle [radius=\rr pt];
	\draw[thin, fill=blue!15,rotate=\a] (1.5*\ss,0.25*\ss) circle [radius=\rr pt];
	\draw[thin, fill=blue!15,rotate=\a] (1.5*\ss,-0.25*\ss) circle [radius=\rr pt];
	\draw[thin, fill=blue!15,rotate=\a] (0.75*\ss,0.5*\ss) circle [radius=\rr pt];
	\draw[thin, fill=blue!15,rotate=\a] (\ss,0.5*\ss) circle [radius=\rr pt];
	\draw[thin, fill=blue!15,rotate=\a] (1.25*\ss,0.5*\ss) circle [radius=\rr pt];
	\draw[thin, fill=blue!15,rotate=\a] (0.75*\ss,-0.5*\ss) circle [radius=\rr pt];
	\draw[thin, fill=blue!15,rotate=\a] (\ss,-0.5*\ss) circle [radius=\rr pt];
	\draw[thin, fill=blue!15,rotate=\a] (1.25*\ss,-0.5*\ss) circle [radius=\rr pt];
	\draw[thin, fill=blue!15,rotate=\a] (\ss,0.75*\ss) circle [radius=\rr pt];
	\draw[thin, fill=blue!15,rotate=\a] (\ss,-0.75*\ss) circle [radius=\rr pt];
}
\node at (0,0) [above right] {$1$};
\node at (\ss,0) [above right] {$g_1$};
\node at (0,\ss) [above right] {$g_2$};
\node at (0.5*\ss,0) [below] {$e_1$};
\node at (0,-2.25*\ss) {$X$};
\end{scope}
\begin{scope}[yshift=-9cm]
\foreach \a in {0,90,180,270} {
	\draw[rotate=\a] (0,0) -- (2*\ss,0);
	\draw[rotate=\a] (\ss,0.5*\ss) -- (\ss,-0.5*\ss);
}
\draw (\ss+0.5*\xx*\ss,0.5*\xx*\ss) -- (\ss,0) -- (\ss+0.5*\xx*\ss,-0.5*\xx*\ss);
\draw (-\ss+0.5*\xx*\ss,0.5*\xx*\ss) -- (-\ss,0) -- (-\ss+0.5*\xx*\ss,-0.5*\xx*\ss);
\draw (1.5*\ss,1.5*\ss) -- (0,0) -- (1.5*\ss,-1.5*\ss);
\draw (0.5*\xx*\ss,\ss+0.5*\xx*\ss) -- (0,\ss) -- (0.5*\xx*\ss,\ss-0.5*\xx*\ss);
\draw (0.5*\xx*\ss,-\ss+0.5*\xx*\ss) -- (0,-\ss) -- (0.5*\xx*\ss,-\ss-0.5*\xx*\ss);
\draw (1.5*\ss,0.25*\ss) -- (1.5*\ss,-0.25*\ss) (1.5*\ss+0.25*\xx*\ss,0.25*\xx*\ss) -- (1.5*\ss,0) -- (1.5*\ss+0.25*\xx*\ss,-0.25*\xx*\ss);
\draw (-1.5*\ss,0.25*\ss) -- (-1.5*\ss,-0.25*\ss) (-1.5*\ss+0.25*\xx*\ss,0.25*\xx*\ss) -- (-1.5*\ss,0) -- (-1.5*\ss+0.25*\xx*\ss,-0.25*\xx*\ss);
\draw (0.8*\ss,1.2*\ss) --(1.6*\ss,1.2*\ss) (1.2*\ss+0.4*\ss,1.2*\ss+0.2*\ss) -- (1.2*\ss,1.2*\ss) -- (1.2*\ss+0.4*\ss,1.2*\ss-0.2*\ss);
\draw[yscale=-1] (0.8*\ss,1.2*\ss) --(1.6*\ss,1.2*\ss) (1.2*\ss+0.4*\ss,1.2*\ss+0.2*\ss) -- (1.2*\ss,1.2*\ss) -- (1.2*\ss+0.4*\ss,1.2*\ss-0.2*\ss);
\draw (-0.25*\ss,1.5*\ss) -- (0.25*\ss,1.5*\ss) (0.25*\xx*\ss,1.5*\ss+0.25*\xx*\ss) -- (0,1.5*\ss) -- (0.25*\xx*\ss,1.5*\ss-0.25*\xx*\ss);
\draw[yscale=-1] (-0.25*\ss,1.5*\ss) -- (0.25*\ss,1.5*\ss) (0.25*\xx*\ss,1.5*\ss+0.25*\xx*\ss) -- (0,1.5*\ss) -- (0.25*\xx*\ss,1.5*\ss-0.25*\xx*\ss);
\draw[thin, fill=blue!15] (0,0) circle [radius=\rr pt];
\draw[thin, fill=blue!15] (1.2*\ss,1.2*\ss) circle [radius=\rr pt];
\draw[thin, fill=blue!15] (1.2*\ss,-1.2*\ss) circle [radius=\rr pt];
\foreach \a in {0,90,180,270} {
	\draw[thin, fill=blue!15,rotate=\a] (\ss,0) circle [radius=\rr pt];
	\draw[thin, fill=blue!15,rotate=\a] (1.5*\ss,0) circle [radius=\rr pt];
}
\node at (0,-2.25*\ss) {$\hX_f$};
\end{scope}
\end{tikzpicture}
\caption{The regular tree of degree 4 with the subset $S$ (shown in red).  The surgered space is the regular tree of degree 6.}
\label{fig:tree}
\end{figure}

Let $T \subset X$ be the set of midpoints of the edges in $S$, considered as a metric space using the metric from $X$.  The map $f \from S \to T$ that sends each edge in $S$ to its midpoint is a pseudo-isometry where $K = 2$ and $C = 1$.  As shown in Figure~\ref{fig:tree}, the surgered space $\hX_f$ is the regular tree of degree 6. 
\end{example}

This example fits into a larger context.  Let $G$ be a finite graph where every edge is isometric to $[0,1]$ and $F \subset G$ a subforest.  In each component $F_i \subseteq F$ fix a point $p_i \in F_i$.  Now consider the universal cover $X$ of $G$ and let $S$ be the full pre-image of $F$ and $T$ the full pre-image of the set $\{p_i\}$.  The natural map $f \from S \to T$ that sends a lift of $F_i$ to the corresponding lift of $p_i$ is a pseudo-isometry.  The surgered space $\hX_f$ is the universal cover of the graph $G'$ obtained by collapsing each component of $F$ to a point.  This set-up for Example~\ref{ex:tree} is shown in Figure~\ref{fig:graph}.      

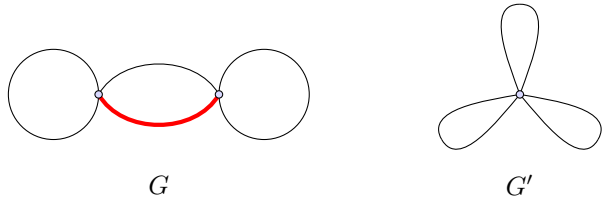
\begin{figure}[ht]
\centering
\begin{tikzpicture}[scale=0.8]
\def\rr{1.75}
\draw (-0.75,0) circle [radius=0.75];
\draw (2.75,0) circle [radius=0.75];
\draw (0,0) to[out=60,in=120] (2,0);
\draw[red,ultra thick] (0,0) to[out=-60,in=-120] (2,0);
\node at (1,-1.5) {$G$};
\draw[thin, fill=blue!15] (0,0) circle [radius=\rr pt];
\draw[thin, fill=blue!15] (2,0) circle [radius=\rr pt];
\begin{scope}[xshift=7cm]
\foreach \a in {0,120,240}{
	\draw[rotate=\a] (0,0) to[out=70,in=0] (0,1.5);
	\draw[rotate=\a] (0,0) to[out=110,in=180] (0,1.5);
}
\draw[thin, fill=blue!15] (0,0) circle [radius=\rr pt];
\node at (0,-1.5) {$G'$};
\end{scope}
\end{tikzpicture}
\caption{The collapse map from $G \to G'$ lifts to the map described in Example~\ref{ex:tree}.  The forest $F \subseteq G$ is indicated in red.}\label{fig:graph}
\end{figure}

\begin{example}\label{ex:all-trees}
    Let $X$ be a regular tree of degree $n$ where every edge is isometric to the unit interval $[0,1]$.  As we show, a \textit{partial fold} of two edges of $X$ is a pseudo-isometry.  This produces a new tree whose vertices have degree either three or $n-1$ and applying this map iteratively produces a tree of degree three.  Because each stage of the map is a pseudo-isometry, we get a pseudo-isometry from any regular tree to a tree of degree three.  Moreover, it is quite clear that regularity of $X$ is not required in this construction.  

    At each vertex choose two edges $u$ and $v$. Since each edge is isometric to $[0,1]$, assume the vertex shared by $u$ and $v$ corresponds to $0$.  Let $u_0$ and $v_0$ denote the \textit{half-edges} of $u$ and $v$, respectively, that are identified with $[0,\frac{1}{2}]$.  The union of each pair of half-edges is a \textit{half-corner}.  Let $S \subset X$  be the set of all the half-corners and let $T \subset X$ be the set of all the $u_0$.  
i
    The map $f \from S \to T$ that sends each half-corner in $S$ to its corresponding $u_0$ is a pseudo-isometry where $K = 3$ and $C = 1$.  As suggested in Figure~\ref{fig:ntree1}, the surgered space $\hX_f$ is a tree where every vertex is either degree three or degree $n-1$. By iterating this process on vertices of degree not equal to three, we arrive at a regular tree of degree three, see Figure~\ref{fig:ntree2}.


\end{example}

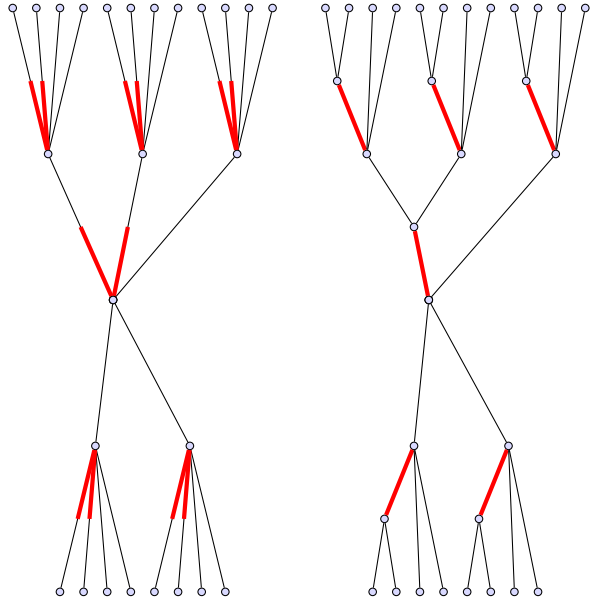
\begin{figure}
\begin{forest}
  for tree={
    draw,             
    circle,           
    minimum size=circle,           
    minimum size=1mm, 
    inner sep=0pt,    
    fill=blue!15,     
    font=\small\bfseries, 
    s sep=2mm,        
    l sep=5mm,       
    anchor=center,    
    if level=1{no edge, before computing xy={l=0,s=0}}{}
  }
  [,fill=blue!50, baseline, name=r
    [, for tree={grow'=north}, label=center:{}, name=s
        [,phantom,tier=pone
           [,phantom, tier=ptwo
             [,phantom, tier=pthree
                [,phantom, tier=pfour
           ]]]]
        [,label=center:{}, tier = ptwo,name=a
          [,label=center:{}, tier=pfour,name=aa]
          [,label=center:{}, tier=pfour,name=aaa]
          [,label=center:{},tier=pfour]
          [,label=center:{},tier=pfour]
        ]
        [,label=center:{},   tier = ptwo,name=b
          [,label=center:{}, tier=pfour,name=bb]
          [,label=center:{}, tier=pfour,name=bbb]
          [,label=center:{},tier=pfour]
          [,label=center:{},tier=pfour]
        ]
        [,label=center:{}, tier = ptwo,name=c
          [,label=center:{},  tier=pfour,name=cc]
          [,label=center:{},  tier=pfour,name=ccc]
          [,label=center:{},tier=pfour]
          [,label=center:{},tier=pfour]
        ]
    ]
    [,for tree={grow=south},label=center:{}
         [,phantom,tier=mone
           [,phantom, tier=mtwo
             [,phantom, tier=mthree
               [,phantom, tier=mfour]
           ]]]
        [, label=center:{}, name=d, tier=mtwo
          [,label=center:{}, tier=mfour, name=dd]
          [,label=center:{}, tier=mfour, name=ddd]
          [,label=center:{},tier=mfour]
          [,label=center:{},tier=mfour]
        ]
        [, label=center:{}, name=e, tier=mtwo
          [,label=center:{},  tier=mfour,name=ee]
          [,label=center:{},  tier=mfour,name=eee]
          [,label=center:{},tier=mfour]
          [,label=center:{},tier=mfour]
          [,phantom]
        ]
    ]
]
    \path [-,red,ultra thick]
          (d) edge ($(d)!0.5!(dd)$)
          (d) edge ($(d)!0.5!(ddd)$)
          (e) edge ($(e)!0.5!(ee)$)
          (e) edge ($(e)!0.5!(eee)$)
          (c) edge ($(c)!0.5!(cc)$)
          (c) edge ($(c)!0.5!(ccc)$)
          (b) edge ($(b)!0.5!(bbb)$)
          (b) edge ($(b)!0.5!(bb)$)
          (r) edge ($(r)!0.5!(a)$)
          (r) edge ($(r)!0.5!(b)$)
          (a) edge ($(a)!0.5!(aaa)$)
          (a) edge ($(a)!0.5!(aa)$);
\end{forest}
\quad
\begin{forest}
  for tree={
    draw,             
    circle,           
    minimum size=1mm, 
    inner sep=0pt,    
    fill=blue!15,     
    font=\small\bfseries, 
    s sep=2mm,        
    l sep=5mm,       
    anchor=center,    
    if level=1{no edge, before computing xy={l=0,s=0}}{}
  }
 [,fill=blue!50, baseline
   [, for tree={grow'=north}, label=center:{}
        [,phantom,tier=pone
            [,phantom, tier=ptwo
              [,phantom, tier=pthree
                [,phantom, tier=pfour
             ]]]] 
        [,label=center:{}, edge={red,ultra thick}
            [, label=center:{}, tier = pone
                [,label=center:{}, edge={red,ultra thick}, tier=ptwo
                [,label=center:{}, tier=pthree]
                [,label=center:{},  tier=pthree]
              ]
              [,label=center:{},tier=pthree]
              [,label=center:{},tier=pthree]
            ]
            [, label=center:{}, tier = pone
                  [,label=center:{}, edge={red,ultra thick}, tier = ptwo
                  [,label=center:{}, tier=pthree]
                  [,label=center:{}, tier=pthree]
              ]
              [,label=center:{},tier=pthree]
              [,label=center:{},tier=pthree]
            ]  
        ]
        [, label=center:{}, tier = pone
              [,label=center:{}, edge={red,ultra thick}, tier=ptwo
              [,label=center:{}, tier=pthree]
              [,label=center:{}, tier=pthree]
          ]
          [,label=center:{},tier=pthree]
          [,label=center:{},tier=pthree]
        ] 
    ]
    [,for tree={grow=south},label=center:{}
        [,phantom,tier=mone
           [,phantom, tier=mtwo
             [,phantom, tier=mthree
                [,phantom, tier=mfour
           ]]]]
          [,label=center:{}, tier=mtwo
               [,label=center:{}, edge={red,ultra thick}, tier=mthree
                   [,label=center:{}, tier=mfour]
                   [,label=center:{}, tier=mfour]
               ]
               [,label=center:{},tier=mfour]
               [,label=center:{},tier=mfour]
           ]
         [,label=center:{}, tier=mtwo
                [,label=center:{}, edge={red,ultra thick}, tier=mthree
                    [,label=center:{}, tier=mfour]
                    [,label=center:{}, tier=mfour]
                ]
                [,label=center:{},tier=mfour]
                [,label=center:{},tier=mfour]
        ]
    ]
]
\end{forest}
\caption{Beginning with a 5-tree, a partial fold of red edges producing vertices of degree either three or four.} 
\label{fig:ntree1}
\end{figure}

\begin{figure}
\begin{forest}
  for tree={
    draw,             
    circle,           
    minimum size=1mm, 
    inner sep=0pt,    
    fill=blue!15,     
    font=\small\bfseries, 
    s sep=2mm,        
    l sep=5mm,       
    anchor=center,    
    if level=1{no edge, before computing xy={l=0,s=0}}{}
  }
 [,fill=blue!50, baseline, name=r
   [, for tree={grow'=north}, label=center:{} 
        [,phantom,tier=pone
            [,phantom, tier=ptwo
              [,phantom, tier=pthree
                [,phantom, tier=pfour
             ]]]] 
        [,label=center:{}, 
        [, label=center:{}, tier = pone, name=f
                [,label=center:{}, tier=ptwo
                [,label=center:{}, tier=pthree]
                [,label=center:{}, tier=pthree]
              ]
              [,label=center:{},tier=pthree, name=ff]
              [,label=center:{},tier=pthree, name=fff]
            ]
            [, label=center:{},      tier = pone, name=g
                  [,label=center:{}, tier = ptwo, 
                  [,label=center:{}, tier=pthree]
                  [,label=center:{}, tier=pthree]
              ]
              [,label=center:{},tier=pthree, name=gg]
              [,label=center:{},tier=pthree, name=ggg]
            ]  
        ]
        [, label=center:{}, tier = pone, name=h
              [,label=center:{}, tier=ptwo
              [,label=center:{}, tier=pthree, ]
              [,label=center:{}, tier=pthree, ]
          ]
          [,label=center:{},tier=pthree, name=hh]
          [,label=center:{},tier=pthree, name=hhh]
        ] 
    ]
    [,for tree={grow=south},label=center:{}, name=rs
         [,phantom,tier=mone
            [,phantom, tier=mtwo
            [,phantom, tier=mthree,
            [,phantom, tier=mfour]]]]
        [, label=center:{}, name=i, tier=mtwo
              [,label=center:{},    tier=mthree, name=ii
                  [,label=center:{},tier=mfour]
                  [,label=center:{},tier=mfour]
              ]
              [,label=center:{},tier=mfour, name=iii]
              [,label=center:{},tier=mfour, name=iiii]
        ]
        [, label=center:{}, name=j, tier=mtwo
           [,label=center:{},     tier=mthree
               [,label=center:{}, tier=mfour]
               [,label=center:{}, tier=mfour]
           ]
           [,label=center:{},tier=mfour, name=jj]
           [,label=center:{},tier=mfour, name=jjj]
                ]
            ]
]
   \path [-,red,ultra thick]
          (j) edge ($(j)!0.5!(jj)$)
          (j) edge ($(j)!0.5!(jjj)$)
          (rs) edge ($(rs)!0.5!(j)$)
          (rs) edge ($(rs)!0.5!(i)$)
          (i) edge ($(i)!0.5!(iii)$)
          (i) edge ($(i)!0.5!(iiii)$)
          (h) edge ($(h)!0.5!(hh)$)
          (h) edge ($(h)!0.5!(hhh)$)
          (g) edge ($(g)!0.5!(ggg)$)
          (g) edge ($(g)!0.5!(gg)$)
          (f) edge ($(f)!0.5!(ff)$)
          (f) edge ($(f)!0.5!(fff)$);
\end{forest}
\quad
\begin{forest}
  for tree={
    draw,             
    circle,           
    minimum size=1mm, 
    inner sep=0pt,    
    fill=blue!15,     
    font=\small\bfseries, 
    s sep=2mm,        
    l sep=5mm,       
    anchor=center,    
    if level=1{no edge, before computing xy={l=0,s=0}}{}
  }
 [,fill=blue!50, baseline
   [, for tree={grow'=north}, label=center:{} 
        [,phantom,tier=pone
            [,phantom, tier=ptwo
              [,phantom, tier=pthree
                [,phantom, tier=pfour
             ]]]] 
        [,label=center:{}, 
        [, label=center:{}, tier = pone, 
                [,label=center:{}, tier=ptwo
                [,label=center:{}, tier=pthree]
                [,label=center:{}, tier=pthree]
              ]
              [,label=center:{}, tier=ptwo, edge = {red,ultra thick}
              [,label=center:{},tier=pthree, ]
              [,label=center:{},tier=pthree, ]
              ]
            ]
            [, label=center:{},      tier = pone, 
                  [,label=center:{}, tier = ptwo, 
                  [,label=center:{}, tier=pthree]
                  [,label=center:{}, tier=pthree]
              ]
              [,label=center:{}, tier=ptwo, edge = {red,ultra thick}
              [,label=center:{},tier=pthree, ]
              [,label=center:{},tier=pthree, ]
              ]
            ]  
        ]
        [, label=center:{}, tier = pone, 
              [,label=center:{}, tier=ptwo
              [,label=center:{}, tier=pthree, ]
              [,label=center:{}, tier=pthree, ]
          ]
        [,label=center:{}, tier=ptwo, edge = {red,ultra thick}
          [,label=center:{},tier=pthree, ]
          [,label=center:{},tier=pthree, ]
        ]
        ] 
    ]
    [,for tree={grow=south},label=center:{}
    [, label=center:{},  edge = {red, ultra thick}
         [, label=center:{},  
         [,phantom,tier=mone
                 [,phantom, tier=mtwo
                   [,phantom, tier=mthree
                     [,phantom, tier=mfour
                 ]]]]
                [,label=center:{},    tier=mone
                    [,label=center:{},tier=mtwo]
                    [,label=center:{},tier=mtwo]
                ]
              [,label=center:{}, tier=mone, edge = {red,ultra thick}
                [,label=center:{},tier=mtwo, ]
                [,label=center:{},tier=mtwo, ]
             ]
         ]
         [, label=center:{},  
            [,label=center:{},     tier=mone, 
                [,label=center:{}, tier=mtwo]
                [,label=center:{}, tier=mtwo]
            ]
         [,label=center:{}, tier=mone, edge = {red,ultra thick}
            [,label=center:{},tier=mtwo, ]
            [,label=center:{},tier=mtwo, ]
            ]
         ]
    ]
    ]
]
\end{forest}
\caption{In each step of the iteration, $S$ and $T$ are redefined.  Continuing the example above, we apply the map again and produce a 3-tree.}
\label{fig:ntree2}
\end{figure}
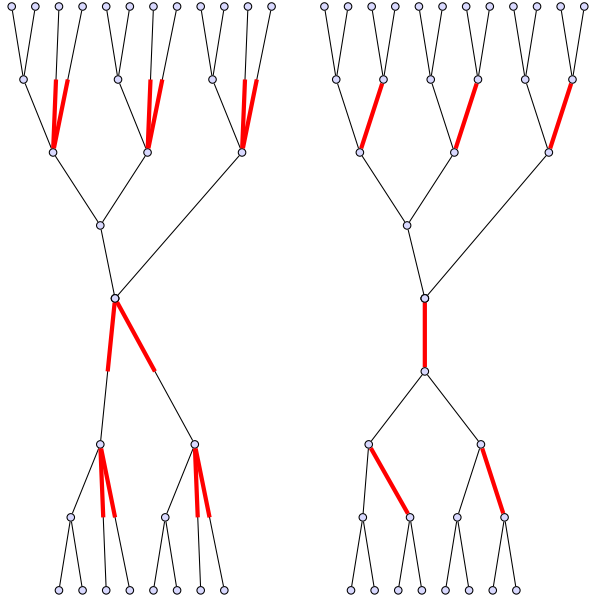

\begin{example}\label{ex:strong}
Let $X$ be the closed positive ray in $\RR$, i.e., $X = [0,\infty)$.  Define two sequences $(a_i), (b_i) \subset X$ as follows starting with $i=0$:
\begin{align*}
a_i &= 1 + i - \left(\sfrac12\right)^i: 0, \sfrac{3}{2}, \sfrac{11}{4}, \sfrac{31}{8}, \ldots \\
b_i &= 2 + i - \left(\sfrac12\right)^i: 1, \sfrac{5}{2}, \sfrac{15}{4}, \sfrac{39}{8}, \ldots
\end{align*}  
These sequences are related by the equations $a_i = b_{i-1} + \left(\sfrac{1}{2}\right)^i$ and $b_i = a_i + 1$.  We consider the subset $S \subset X$ given by:
\begin{equation*}
S = \bigcup_{i=0}^\infty [a_i,b_i].
\end{equation*}
See Figure~\ref{fig:strong}. 

\begin{figure}[ht]
\scalebox{.88}{
\begin{tikzpicture}
\def\s{1.75}
\node at (-1,0) {$X:$};
\draw[very thick] (0,0) -- (6*\s,0);
\filldraw (6*\s + 0.2,0) circle [radius=1pt];
\filldraw (6*\s + 0.35,0) circle [radius=1pt];
\filldraw (6*\s + 0.5,0) circle [radius=1pt];
\draw[line width=3pt,red] (0,0) -- (\s,0) (1.5*\s,0) -- (2.5*\s,0) (2.75*\s,0) -- (3.75*\s,0) (3.8725*\s,0) -- (4.8725*\s,0) (4.96875*\s,0) -- (5.96875*\s,0);
\draw (0,-0.15) -- (0,0.15) (\s,-0.15) -- (\s,0.15);
\draw (1.5*\s,-0.15) -- (1.5*\s,0.15) (2.5*\s,-0.15) -- (2.5*\s,0.15);
\draw (2.75*\s,-0.15) -- (2.75*\s,0.15) (3.75*\s,-0.15) -- (3.75*\s,0.15);
\draw (3.8725*\s,-0.15) -- (3.8725*\s,0.15) (4.8725*\s,-0.15) -- (4.8725*\s,0.15);
\draw (4.96875*\s,-0.15) -- (4.96875*\s,0.15) (5.96875*\s,-0.15) -- (5.96875*\s,0.15);
\node at (0+0.1,0) [label=below:{$a_0$}] {};
\node at (\s+0.1,0) [label=below:{$b_0$}] {};
\node at (1.5*\s+0.1,0) [label=above:{$a_1$}] {};
\node at (2.5*\s+0.05,0) [label=above:{$b_1$}] {};
\node at (2.75*\s+0.1,0) [label=below:{$a_2$}] {};
\node at (3.75*\s+0.1,0) [label=below:{$b_2$}] {};
\node at (3.8725*\s+0.1,0) [label=above:{$a_3$}] {};
\node at (4.8725*\s+0.05,0) [label=above:{$b_3$}] {};
\node at (4.96875*\s+0.1,0) [label=below:{$a_4$}] {};
\node at (5.96875*\s+0.1,0) [label=below:{$b_4$}] {};
\end{tikzpicture}
}
\caption{The closed positive ray with the subset $S$ (shown in red).  The surgered space is the half-open interval $[0,1) \subset \RR$.}\label{fig:strong}
\end{figure}
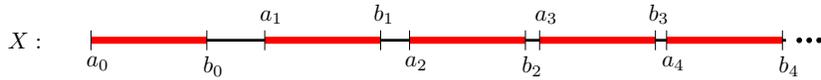

    Let $T \subset X$ be the set of all $a_i$, considered as a metric space using the metric from $X$.  The map $f \from S \to T$ that sends each interval $[a_i,b_i]$ to $a_i$ is a (continuous) quasi-isometry with $K = 1, C = 2$, but it is not a pseudo-isometry for any choice of constants $K$ and $C$.  Indeed, $d_S(b_{i-1},a_i) = \left(\sfrac{1}{2}\right)^i \to 0$ as $i \to \infty$  while $d_T(f(b_{i-1}),f(a_i)) = d_T(a_{i-1},a_i) \geq 1$ for all $i$.  The map collapses all but intervals of the form $[b_n,a_{n+1}]$, each having length $(\frac{1}{2})^{n+1}$.  Since $\sum_0^\infty  (\frac{1}{2})^{n+1} = 1$, the surgered space $\hX_f$ is isometric to the half-open interval $[0,1) \subset \RR$.  In particular, $X$ is not pseudo-isometric (nor quasi-isometric) to $\hX_f$.

	The failure of this example is directly tied to the failure for Lemma~\ref{lem:admissible length lower bound} in this setting.  There are admissible sequences with bounded length connecting points arbitrarily far apart in $X$.  See Remark~\ref{rmk:admissible length lower bound}.
\end{example}

This example shows the necessity of the hypothesis of a pseudo-isometry in the statement of Theorem~\ref{thm:main}.  It would be interesting to find robust conditions on $S \subset X$ so that if $f \from S \to T$ is a quasi-isometry, then the surgered space $\hX_f$ is quasi-isometric to $X$.

\bibliographystyle{vancouver}
\bibliography{pseudo-isometry}

\begin{thebibliography}{10}

\bibitem{Milnor2}
Milnor J.
\newblock A procedure for killing homotopy groups of differentiable manifolds.
\newblock In: Proc. {S}ympos. {P}ure {M}ath., {V}ol. {III}. Amer. Math. Soc.,
  Providence, RI; 1961. p. 39--55.

\bibitem{Gordon}
Mrowka TS, Ozsv\'{a}th PS, editors.
\newblock Low dimensional topology. vol.~15 of IAS/Park City Mathematics
  Series.
\newblock American Mathematical Society, Providence, RI; 2009.
\newblock Lecture notes from the 15th Park City Mathematics Institute (PCMI)
  Graduate Summer School held in Park City, UT, Summer 2006.
\newblock Available from: \url{https://doi.org/10.1090/pcms/015}.

\bibitem{Stillwell}
Stillwell J.
\newblock Poincar\'{e} and the early history of 3-manifolds.
\newblock Bull Amer Math Soc (NS). 2012;49(4):555--576.
\newblock Available from:
  \url{https://doi.org/10.1090/S0273-0979-2012-01385-X}.

\bibitem{Rolfsen}
Rolfsen D.
\newblock Knots and links. vol.~7 of Mathematics Lecture Series.
\newblock Publish or Perish, Inc., Houston, TX; 1990.
\newblock Corrected reprint of the 1976 original.

\bibitem{Svarc}
\v{S}varc AS.
\newblock A volume invariant of coverings.
\newblock Doklady Akademii Nauk SSSR. 1955;3:32 --34.

\bibitem{Milnor}
Milnor J.
\newblock A note on curvature and fundamental group.
\newblock J Differential Geometry. 1968;2:1--7.
\newblock Available from: \url{http://projecteuclid.org/euclid.jdg/1214501132}.

\bibitem{Gromov}
Gromov M.
\newblock Hyperbolic manifolds, groups and actions.
\newblock In: Riemann surfaces and related topics: {P}roceedings of the 1978
  {S}tony {B}rook {C}onference ({S}tate {U}niv. {N}ew {Y}ork, {S}tony {B}rook,
  {N}.{Y}., 1978). vol. No. 97 of Ann. of Math. Stud. Princeton Univ. Press,
  Princeton, NJ; 1981. p. 183--213.

\bibitem{Clay}
Clay M, Margalit D.
\newblock Office Hours with a Geometric Group Theorist.
\newblock Princeton University Press; 2017.

\bibitem{Brid-Helf}
Bridson MR, Haefliger A.
\newblock Metric spaces of non-positive curvature. vol. 319 of Grundlehren der
  mathematischen Wissenschaften [Fundamental Principles of Mathematical
  Sciences].
\newblock Springer-Verlag, Berlin; 1999.
\newblock Available from: \url{https://doi.org/10.1007/978-3-662-12494-9}.

\bibitem{Mostow}
Mostow GD.
\newblock Strong rigidity of locally symmetric spaces. vol. No. 78 of Annals of
  Mathematics Studies.
\newblock Princeton University Press, Princeton, NJ; University of Tokyo Press,
  Tokyo; 1973.

\bibitem{Thurston}
Thurston WP.
\newblock Three-dimensional geometry and topology. {V}ol. 1. vol.~35 of
  Princeton Mathematical Series.
\newblock Princeton University Press, Princeton, NJ; 1997.

\end{thebibliography}

\end{document}